\documentclass{article}
\usepackage{amsmath,amsfonts,amssymb,amsthm} 
\usepackage{graphicx}
\usepackage[T1]{fontenc}
\usepackage[utf8]{inputenc}
\usepackage{multirow}
\usepackage{verbatim}
\usepackage{multicol}

\newtheorem{theorem}{Theorem}[section]
\newtheorem{observation}[theorem]{Observation}

\newtheorem{conjecture}[theorem]{Conjecture}
\newtheorem{example}{Example}[section]
\newtheorem{lemma}[theorem]{Lemma}

\theoremstyle{definition}
\newtheorem{definition}{Definition}[section]

\newcommand{\G}{\mathcal{G}}

\DeclareMathOperator{\td}{\bf{td}}
\DeclareMathOperator{\obs}{\bf{obs}}
\newcommand{\ind}{_\sqsubseteq}
\newcommand{\sub}{_\subseteq}
\newcommand{\minor}{_\leq}

\title{Uniqueness and minimal obstructions for tree-depth}
\author{Michael D. Barrus\thanks{Department of Mathematics, University of Rhode Island, Kingston, Rhode Island 02881, United States; email: \texttt{barrus@uri.edu}} ~and
John Sinkovic\thanks{Department of Mathematics and Statistics, Georgia State University, Atlanta, Georgia 30302, United States; email: \texttt{johnsinkovic@gmail.com}}}

\begin{document}

\maketitle

\begin{abstract}
A $k$-ranking of a graph $G$ is a labeling of the vertices of $G$ with values from $\{1,\dots,k\}$ such that any path joining two vertices with the same label contains a vertex having a higher label. The tree-depth of $G$ is the smallest value of $k$ for which a $k$-ranking of $G$ exists. The graph $G$ is $k$-critical if it has tree-depth $k$ and every proper minor of $G$ has smaller tree-depth.

We establish partial results in support of two conjectures about the order and maximum degree of $k$-critical graphs. As part of these results, we define a graph $G$ to be \emph{1-unique} if for every vertex $v$ in $G$, there exists an optimal ranking of $G$ in which $v$ is the unique vertex with label 1. We show that several classes of $k$-critical graphs are $1$-unique, and we conjecture that the property holds for all $k$-critical graphs. Generalizing a previously known construction for trees, we exhibit an inductive construction that uses $1$-unique $k$-critical graphs to generate large classes of critical graphs having a given tree-depth.

\medskip
\emph{Keywords:} Graph minors, tree-depth, vertex ranking
\end{abstract}


\section{Introduction}


The \emph{tree-depth} of a graph $G$, denoted $\td(G)$, is defined as the smallest natural number $k$ such that the vertices of $G$ may be labeled with elements of $\{1,\dots,k\}$ such that every path joining two vertices with the same label contains a vertex having a larger label. The name of this parameter refers to its equivalence with the minimum height of a rooted forest $F$ with the same vertex set of $G$ for which each edge of $G$ either belongs to $F$ or joins vertices having an ancestor--descendant relationship in $F$~\cite[Definition 6.1]{NesetrilOssonadeMendez12}. Tree-depth has also been referred to as the ordered chromatic number~\cite{BarNoyEtAl12,path} or vertex ranking number~\cite{BodlaenderEtAl98,ChangEtAl10,IyerEtAl88}. (See~\cite{NesetrilOssonadeMendez12,NesetrilOssonadeMendez06,NesetrilOssonadeMendez08} and the references cited above for further results and references.)

While much is known about the computational complexity of determining the tree-depth of a graph~\cite{NesetrilOssonadeMendez12,Pothen88}, from a structural standpoint we wish to understand what ``causes'' a given graph to have a particular tree-depth. In particular, since $\td(G)$ is defined as a minimum, what obstructions prevent $G$ from having a smaller tree-depth? One answer lies in the minors of $G$; as noted in~\cite[Lemma 6.2]{NesetrilOssonadeMendez12}, $\td(G) \geq \td(H)$ whenever $H$ is a minor of $G$. Define a graph $M$ to be \emph{critical} if every proper minor of $M$ has tree-depth less than $\td(M)$. If $\td(G)=k$ for a particular $k$, then we may attribute this to the fact that $G$ contains a critical minor with tree-depth $k$, and $G$ contains no critical minor with tree-depth $k+1$.

(Note that in other places in the literature, ``critical'' has sometimes been used to describe graphs for which every proper \emph{subgraph} has a smaller tree-depth; here we refer to these graphs as \emph{subgraph-critical}. If a graph has the property that every proper induced subgraph has a smaller tree-depth, then the graph will be called \emph{induced-subgraph-critical}. Since the minor relation encompasses more than the subgraph relation, it will be more natural here to refer to minors when using the unqualified term ``critical.'')

Curiosity about the critical graphs has begun to generate both structural results and questions. Notably, in~\cite{DGT} (see also~\cite{GT09}), Dvo\u{r}\'{a}k, Giannopoulou, and Thilikos defined $\mathcal{G}_k$ to be the class of graphs having tree-depth at most $k$, and $\obs\minor(\G_k)$ to be the set of minimal graphs under the minor-containment order having tree-depth greater than $k$ (in our terminology, $\obs\minor(\G_k)$ consists of all critical graphs with tree-depth $k+1$). Among other things, the paper~\cite{DGT} presented the elements of $\obs\minor(\G_k)$ for $k \in \{1,2,3\}$ (see Figure~\ref{fig:minorminimallist}). The authors also gave a constructive result. %
\begin{figure}[htb]
       \center{\includegraphics{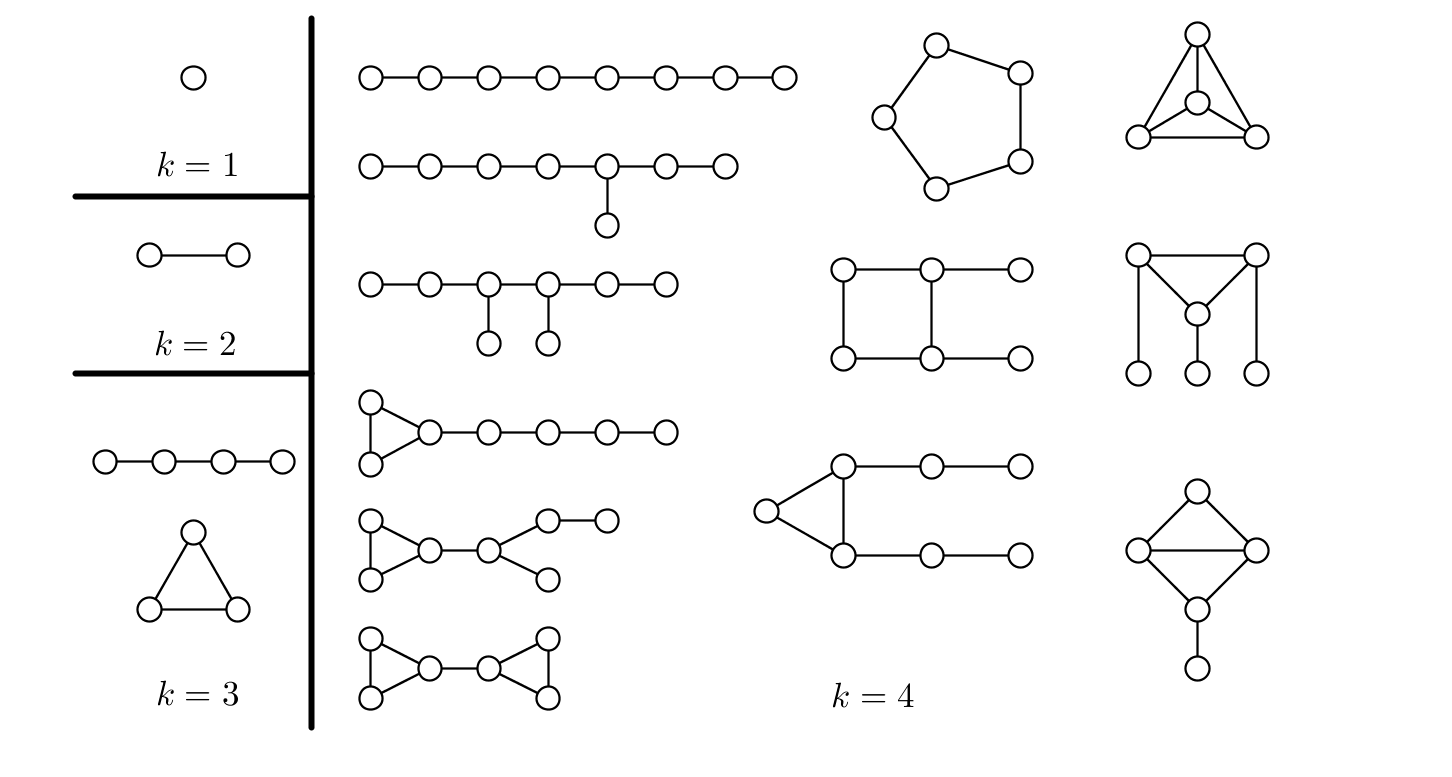}}
       \caption{\label{fig:minorminimallist} $k$-critical graphs for $k \in \{1,2,3,4\}$}
\end{figure} %
\begin{theorem}[\cite{DGT}]\label{thm: DGT edge}
	Given vertex disjoint graphs $G$ and $H$ in $\obs\minor(\G_{k})$, if a graph $J$ is formed by adding to the disjoint union $G+H$ an edge having one endpoint in $G$ and the other in $H$, then $J$ belongs to $\obs\minor(\G_{k+1})$.
\end{theorem}
It is easy to see that the sizes of the classes $\obs\minor(\G_k)$ mushroom as $k$ increases; the paper~\cite{DGT} uses Theorem~\ref{thm: DGT edge} to give a lower bound on the size of $\obs\minor(\G_4)$ by determining the number of trees in this set.

Closer examination of Figure~\ref{fig:minorminimallist} suggests structural properties that may possibly hold for all critical graphs. In this paper we address two particular conjectures along these lines. The first, which appears in~\cite{DGT}, deals with the orders of critical graphs; in its original form the conjecture is extended to all induced-subgraph-critical graphs.

\begin{conjecture} \label{conj: critical graph orders}
	Every critical graph with tree-depth $k$ has at most $2^{k-1}$ vertices.
\end{conjecture}

The second conjecture, which does not seem to have appeared yet in the literature, concerns vertex degrees.

\begin{conjecture} \label{conj: max degree}
	Every critical graph with tree-depth $k$ has maximum degree at most $k-1$.
\end{conjecture}

Conjecture~\ref{conj: max degree} is easily proved for a large class of critical graphs, and this class will be the focus of this paper. As we will see, the class may in fact include all critical graphs. We begin with some definitions.

Given a graph $G$, we will call a labeling of the vertices of $G$ with labels from $\{1,\dots,k\}$ a \emph{feasible} labeling if every path in $G$ between two vertices with the same label ($\ell$, say) passes through a vertex with a label greater than $\ell$. Adopting terminology from previous authors, we call a feasible labeling with labels from $\{1,\dots,k\}$ a \emph{($k$-)ranking of $G$}, and we refer to the labels as \emph{ranks} or \emph{colors} (note that every feasible labeling is a proper coloring of $G$). We call a ranking of $G$ \emph{optimal} if it is a $\td(G)$-ranking. A critical graph with tree-depth $k$ will be called $k$-critical.

\begin{definition}
A graph $G$ is \emph{$1$-unique} if for every vertex $v$ of $G$ there is an optimal ranking of $G$ in which vertex $v$ is the only vertex receiving rank 1.
\end{definition}

For example, the graph $P_8$ is 1-unique, as the rankings in Figure~\ref{fig:P8} and their reflections about the center of the path show.

\begin{figure}
\centering
\includegraphics[scale=.3]{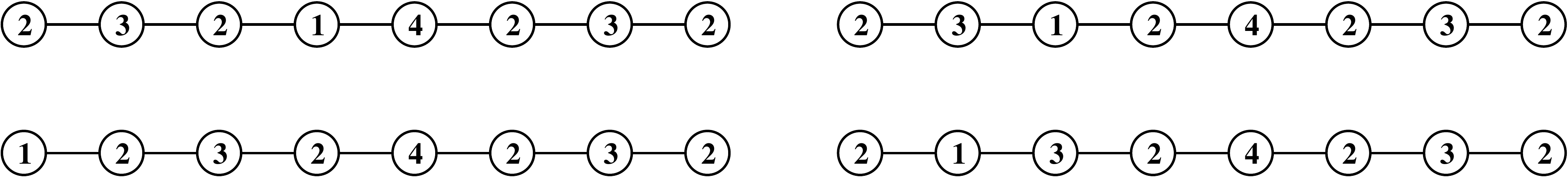}
\caption{Rankings of $P_8$ demonstrating $1$-uniqueness.}
\label{fig:P8}
\end{figure}

For 1-unique graphs, Conjecture~\ref{conj: max degree} is true. Indeed, note that in any feasible ranking of an arbitrary graph $G$, the neighbors of a vertex receiving rank 1 must all receive distinct ranks greater than 1; otherwise, some path joining vertices with the same rank would not contain an intermediate vertex with higher rank. This implies that in an optimal ranking any vertex with rank 1 has degree at most $\td(G)-1$. It immediately follows that $1$-unique graphs have maximum degree at most 1 less than their tree-depths.

The graph $P_8$ is both 4-critical and 1-unique; in fact, one can check that all the critical graphs with tree-depth at most $4$ (see Figure~\ref{fig:minorminimallist}) are 1-unique. It is natural to wonder whether there are critical graphs that are not 1-unique. We conjecture that none exist. 
\begin{conjecture} \label{conj: 1-unique}
	All critical graphs are 1-unique.
\end{conjecture}

In this paper we study the property of 1-uniqueness and give partial results towards Conjectures~\ref{conj: critical graph orders} and~\ref{conj: 1-unique} (and hence~\ref{conj: max degree}, which we have seen follows from~\ref{conj: 1-unique}). In Section~\ref{sec: two} we introduce the more general notion of $t$-uniqueness of a graph and show that the $t$-unique graphs for different values of $t$ form nested families, with $1$-unique graphs forming the smallest such class. In Section 3 we establish partial results towards Conjecture~\ref{conj: 1-unique}, showing that the $1$-unique graphs with tree-depth $k$ satisfy many of the same minimality properties that $k$-critical graphs do. In Section 4 we generalize the construction in Theorem~\ref{thm: DGT edge} and show that all graphs inductively constructed in this way, beginning with graphs from families that contain all critical graphs in~\cite{DGT}, are 1-critical, which lends further support for Conjecture~\ref{conj: 1-unique}. We show that the graphs constructed all satisfy Conjecture~\ref{conj: critical graph orders} as well.

Before beginning, we define some terms and notation. Given a graph $G$, let $V(G)$ and $E(G)$ denote its vertex set and edge set, respectively. The \emph{order} of $G$ is given by $|V(G)|$. Given a vertex $v$ of $G$, let $N_G(v)$ denote the neighborhood of $v$ in $G$, and let $G-v$ denote the graph resulting from the deletion of $v$. Similarly, given a set $S \subseteq V(G)$, let $G-S$ denote the graph obtained by deleting all vertices in $S$ from $G$. For $e \in E(G)$, let $G-e$ denote the graph obtained by deleting edge $e$ from $G$. We indicate the disjoint union of graphs $G$ and $H$ by $G+H$, and we indicate a disjoint union of $k$ copies of $G$ by $kG$. We use $\langle p_1,\dots,p_k\rangle$ to denote a path from $p_1$ to $p_k$, with vertices listed in the order the path visits them; the \emph{length} of such a path is $k-1$, the number of its edges. We use $K_n$, $P_n$, and $C_n$ to denote the complete graph, path, and cycle with $n$ vertices, respectively.

\section{$t$-Unique Graphs} \label{sec: two}
As in \cite{DGT}, let $\mathcal{G}_k$ be the class of all graphs with tree-depth at most $k$.  We write $G\leq H$ to indicate that $G$ is a minor of $H$, and we use $\subseteq$ and $\sqsubseteq$, respectively, to denote the subgraph and induced subgraph relations.  For each $\mathsf{R}\in\{\sqsubseteq, \subseteq, \leq\}$, let $\obs_{\mathsf{R}}(\mathcal{G}_k)$ denote the set of graphs not in $\mathcal{G}_k$ that are minimal with respect to $\mathsf{R}$.  Note that $\obs_{\leq}(\mathcal{G}_k)\subseteq\obs_{\subseteq}(\mathcal{G}_k)\subseteq \obs_{\sqsubseteq}(\mathcal{G}_k)$, and that the elements of $\obs_\leq(\G_k)$, $\obs_\subseteq(\G_k)$, and $\obs_\sqsubseteq(\G_k)$, respectively, are precisely the critical, subgraph-critical, and induced-subgraph-critical graphs with tree-depth $k+1$.

We now generalize the definition of 1-uniqueness from the last section. Recall from \cite{DGT} the following observation:

\begin{observation}\label{obs:forbidden} If $G\in \obs_{\sqsubseteq}(\mathcal{G}_k)$ $(\text{or }   \obs_{\subseteq}(\mathcal{G}_k)  \text{ or }  \obs_{\leq}(\mathcal{G}_k))$  then for every $v\in V(G)$ there exists a $(k+1)$-ranking $\rho$ such that $\rho(v)=k+1$.
\end{observation}

In a connected graph, only one vertex receives the highest label in an optimal ranking. For some optimal rankings (as in Figure~\ref{fig:P8}) other values may appear on just one vertex.

\begin{definition}  A vertex $v$ of $G$ is \emph{$t$-unique} if there exists an optimal ranking of $G$ where $v$ is the unique vertex with rank $t$.  The graph $G$ is \emph{$t$-unique} if each of its vertices is $t$-unique.
\end{definition}

The notion of a $t$-unique graph resembles that of a centered coloring. As explained in~\cite{NesetrilOssonadeMendez08}, a \emph{centered coloring} of a graph $G$ is a vertex coloring with the property that in every connected subgraph of $G$ some color appears exactly once. The minimum number of colors necessary for a centered coloring is then $\td(G)$, and an optimal ranking of $G$ is a centered coloring. Similarly, $t$-uniqueness deals with a color appearing once, though by our definition this color is the fixed color $t$, and the only subgraph of $G$ considered is $G$ itself. Furthermore, $t$-uniqueness is a property of a vertex or graph, rather than of a single coloring; for a graph $G$, $t$-uniqueness requires that multiple optimal colorings exist, placing the color $t$ at each vertex of $G$ in turn.

We now study $t$-uniqueness and how it relates to the classes $\obs_{\mathsf{R}}(\mathcal{G}_k)$ for $\mathsf{R} \in \{\sqsubseteq, \subseteq, \leq\}$.

\begin{lemma}\label{lem:inducedsubgraphsunique} Let $\td(G)=k+1$.  Then $G\in \obs\ind(\mathcal{G}_k)$ if and only if $G$ is $(k+1)$-unique.
\end{lemma}

\begin{proof} In any ranking, a connected graph has a unique vertex of highest rank. Thus Observation~\ref{obs:forbidden} is equivalent to the statement that graphs in $\obs_{\sqsubseteq}(\mathcal{G}_k)$ are $(k+1)$-unique.   If $G$ is $(k+1)$-unique, for any vertex $v$ in $G$ there is an optimal ranking $\rho$ for which $v$ is the unique vertex with rank $k+1$.  Since $\td(G)=k+1$, the labeling $\rho$ restricted to $G-v$ is a ranking using fewer than $k+1$ colors.
\end{proof}

The notion of $t$-uniqueness suggests a certain minimality in graphs with respect to tree-depth. As in the proof of Lemma~\ref{lem:inducedsubgraphsunique}, we may begin with an optimal ranking $\rho$ of $G$ that demonstrates the $t$-uniqueness of a vertex $v$ and restrict $\rho$ to $G-v$.  We derive an optimal ranking of $G-v$ with fewer colors by decreasing by 1 each rank of $\rho$ that is greater than $t$. Thus the $t$-unique vertex $v$ is the only impediment to a ranking of the graph using fewer colors. Lemma~\ref{lem:inducedsubgraphsunique} illustrates this type of minimality in graphs in $\obs\ind(\mathcal{G}_k)$, and we will observe a stronger form of it in many (possibly all) graphs in $\obs\minor(\mathcal{G}_k)$.

We now present some results on $t$-uniqueness in graphs.

\begin{lemma}\label{lem:uniquespectrum} If $G$ is $t$-unique, then $G$ is $s$-unique for all $s$ such that $t\leq s\leq \td(G)$.
\end{lemma}

\begin{proof} We show that if a vertex is $k$-unique for some $k\leq \td(G)-1$ then it is $(k+1)$-unique.  Let $\rho$ be an optimal ranking of $G$, and suppose $v$ is the unique vertex with color $k$.  Form $\rho'$ from $\rho$ by reassigning the color $k$ to all vertices $w$ such that $\rho(w)=k+1$ and reassigning $\rho'(v)=k+1$.   Let $x$ and $y$ be vertices of $G$ such that $\rho'(x)=\rho'(y)$; note that $\rho(x)=\rho(y)$ by our construction.  Then $\rho'(x)\neq k+1$ since $v$ is the unique vertex with label $k+1$.  In every $xy$-path there exists a vertex $z$ for which $\rho(z)>\rho(x)$.   If $\rho'(x)>k+1$ or if $\rho'(z)<k$, then \[\rho'(x)=\rho(x)<\rho(z)=\rho'(z).\] Assume now that $\rho'(x)\leq k+1$ and $\rho'(z)\geq k$. If $\rho'(z)\leq \rho'(x)$ then $k\leq \rho'(z)\leq \rho'(x)\leq k+1$. Since $\rho(z)\neq\rho(x)$, we have $\rho'(z)\neq \rho'(x)$ by construction, implying that $\rho'(x)=k+1$, a contradiction. Hence $\rho'(z)>\rho'(x)$ in every case and $\rho'$ is an optimal ranking of $G$ where $v$ is the unique vertex with color $k+1$.
\end{proof}

In light of Lemmas \ref{lem:inducedsubgraphsunique} and \ref{lem:uniquespectrum}, $G\in\obs\ind(\mathcal{G}_k)$ for some $k$ if and only if $G$ is $t$-unique for some $t$.   By Lemma \ref{lem:uniquespectrum} we can group the graphs in $\obs\ind(\mathcal{G}_k)$ by the minimum $t$ for which they are $t$-unique.  The $1$-unique graphs are of particular interest because they satisfy the most restrictive condition.  Let $\mathcal{U}_k$ be the set of all graphs with tree-depth $k$ that are $1$-unique.

Since $\mathcal{U}_{k+1}$ is a subset of $\obs\ind(\mathcal{G}_k)$, it is natural to ask whether it contains or is contained in either $\obs\sub(\mathcal{G}_k)$ or $\obs\minor(\mathcal{G}_k)$.   This is not the case:  Let $G_k$ be the graph obtained from $C_{2^k+1}$ by adding a chord between the neighbors $u,w$ of a vertex $v$.  We will show in the next section that $G_k$ is $1$-unique and that $\td(G_k)=\td(C_{2^k+1})=k+2. $ Since $G_k$ contains $C_{2^k+1}$ as a subgraph, $G_k\not\in \obs\sub(\mathcal{G}_{k+1})$.  We will also show that the graph $C_{2^k+2}$ is in $\obs_{\subseteq}(\mathcal{G}_{k+1})$, but that it is not $1$-unique.  Thus $\obs_{\subseteq}(\mathcal{G}_{k})$ and $\mathcal{U}_{k+1}$ are incomparable for all $k\geq 2$.  However, their intersection is of interest.

\begin{theorem}\label{thm:conjecture} If $G$ is $1$-unique and subgraph-critical with tree-depth $k+1$, then $G$ is $(k+1)$-critical; in symbols, \begin{equation}\label{eq:conjecture} \mathcal{U}_{k+1}\cap \obs\sub(\mathcal{G}_k)\subseteq \obs\minor(\mathcal{G}_k). \end{equation}
\end{theorem}

If Conjecture~\ref{conj: 1-unique} is true, then equality holds in~\eqref{eq:conjecture}.

The proof of Theorem \ref{thm:conjecture} requires several preliminary steps, and we postpone it until the next section. We close this section with a Venn diagram illustrating the relationships between the sets of this section in the following figure; the shaded region indicates the set $\obs\minor(\mathcal{G}_k)$, and the question mark indicates the region that Conjecture~\ref{conj: 1-unique} states is empty.
\begin{figure}[htb]
\centering
\includegraphics{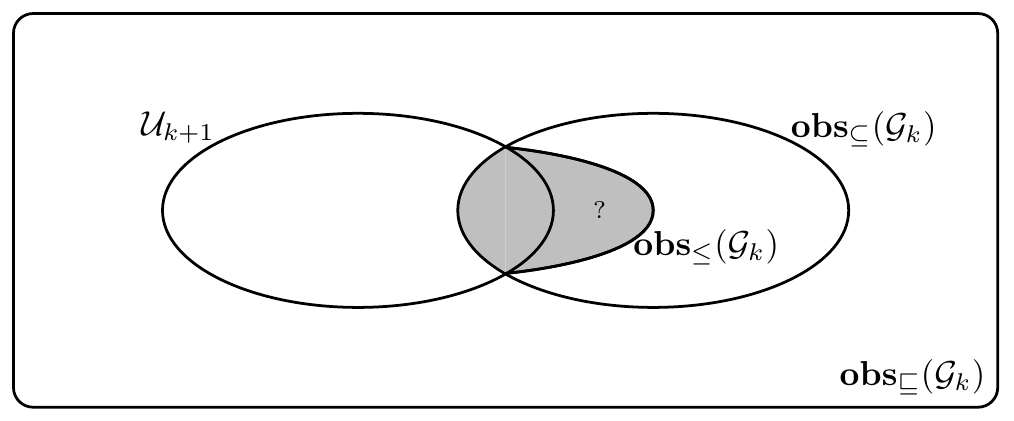}
\caption{Depiction of set intersections}
\label{fig:venndiagram}
\end{figure}

\section{Properties of $1$-unique graphs}

In this section we prove Theorem~\ref{thm:conjecture} and other claims from the previous section and describe further properties of $1$-unique graphs.  We begin with a characterization of $1$-unique vertices. Given a vertex $v$ in a graph $G$, a \emph{star-clique transform at $v$} removes $v$ from $G$ and adds edges between the vertices in $N_G(v)$ so as to make them a clique.

\begin{theorem}\label{thm:starclique} Let $v$ be a vertex of a graph $G$, and let $H$ be the graph obtained through the star-clique transform at $v$ of $G$.  Vertex $v$ is $1$-unique in $G$ if and only if $\td(H)<\td(G)$.
\end{theorem}

Theorem~\ref{thm:starclique} is a consequence of a more general result on tree-depth that we will prove first.

\begin{definition} \label{defn:G S}
Given a graph $G$ and a subset $S$ of its vertices, let $G\langle S \rangle$ denote the graph with vertex set $S$ in which vertices $u$ and $v$ are adjacent if they are adjacent in $G$ or if some component of $G-S$ has a vertex adjacent to $u$ and a vertex adjacent to $v$.
\end{definition}

\begin{theorem}\label{thm:newbound}
If $G$ is a graph, then \[\td(G) = \min_{S \subseteq V(G)} \left(\td(G\langle S \rangle) + \td(G-S)\right).\] Furthermore, $\td(G) = \td(G\langle T \rangle) + \td(G-T)$ if and only if there exists an optimal ranking of $G$ in which the vertices in $T$ receive higher colors than the vertices outside $T$.
\end{theorem}
\begin{proof} %
We show first that for any $S \subseteq V(G)$, we can obtain a ranking of $G$ with $\td(G\langle S\rangle)+\td(G-S)$ colors in which the vertices in $S$ receive strictly higher colors than the vertices outside $S$. Let $\alpha$ and $\beta$ be optimal rankings of $G-S$ and $G\langle S \rangle$, respectively. Construct a ranking $\rho$ of $G$ by defining $\rho(v)=\alpha(v)$ for all $v \in V(G) - S$ and $\rho(w) = \beta(w) + \td(G-S)$ for all $w \in S$.

We claim that $\rho$ is a ranking of $G$. Suppose vertices $x$ and $y$ in $G$ receive the same color $c$, and consider a path $P$ having $x$ and $y$ as its endpoints. Suppose first that $c \leq \td(G-S)$. If the path $P$ includes a vertex of $S$, then $P$ contains a vertex with color greater than $c$, as desired. Otherwise, the path $P$ is contained in $G-S$, and by construction $P$ contains a vertex colored with a value greater than $c$.

If instead $c > \td(G-S)$, then $x$ and $y$ both belong to $S$. Suppose $u,w_1, \dots, w_\ell, v$ is a list of consecutive vertices of $P$ with the property that $u,v \in S$ and $w_1, \dots, w_\ell \notin S$. Note that $uv$ is an edge of $G \langle S \rangle$, so we may form a path $P'$ in $G \langle S \rangle$ simply by removing vertices not in $S$ from the ordered list of vertices of $P$. Since $\beta(x)=\beta(y)$ and $P'$ joins $x$ and $y$ in $G\langle S \rangle$, some vertex of $P'$ receives a higher color than $\beta(x)$ in the ranking $\beta$, and by construction this vertex is a vertex of $P$ receiving a higher color than $c$, as desired.

Thus $\td(G) \leq \min_{S \subseteq V(G)} \left(\td(G\langle S \rangle) + \td(G-S)\right)$, and if $\td(G\langle T\rangle)+\td(G-T)=\td(G)$ for some subset $T$ of $V(G)$, then the coloring described above is an optimal ranking of $G$ in which the vertices of $T$ receive higher colors than those outside $T$.

We now show that if $T \subseteq V(G)$ and there exists an optimal ranking of $G$ in which the vertices in $T$ receive higher colors than the vertices outside $T$, then $\td(G) = \td(G\langle T \rangle) + \td(G-T)$; this demonstrates equality in the inequality from the previous paragraph. Suppose $\tau$ is a $\td(G)$-ranking of $G$ in which $\tau(v) > \tau(w)$ whenever $v \in T$ and $w \notin T$. Let $\beta$ denote the restriction of $\tau$ to $V(G)-T$; since $\tau$ is an optimal coloring, we may assume that $\beta$ is an optimal ranking of $G-T$. Now define a labeling $\alpha$ of the vertices of $T$ by letting $\alpha(v) = \tau(v)-\td(G-T)$. We claim that $\alpha$ is a ranking of $G\langle T \rangle$. Clearly $\alpha(v) \geq 1$ for all $v \in T$. Suppose now that there exist distinct vertices $x$ and $y$ in $T$ such that $\alpha(x)=\alpha(y)$, and let $P$ be a path joining $x$ and $y$ in $G \langle T \rangle$.

For any two adjacent vertices $u$ and $v$ in $G\langle T \rangle$, either $uv$ is an edge of $G$, or there exists a path $\langle w_1,\dots,w_\ell \rangle$ in $G-T$ such that $uw_1$ and $w_{\ell} v$ are edges in $G$. Modify $P$ to obtain a walk $W$ in $G$ by inserting such a path between each pair $u,v$ of consecutive vertices of $P$ that are nonadjacent in $G$. The walk $W$ contains a path $P'$ between $x$ and $y$ in $G$; by assumption, $P'$ contains a vertex $z$ such that $\tau(z)>\tau(x)$. Since $\tau$ assigns larger colors to vertices in $T$ than to vertices not in $T$,  we have $z \in T$. This forces $z$ to be a vertex of $P$, and it follows that $\alpha$ is a ranking of $G \langle T \rangle$. If $\alpha$ were not an optimal ranking of $G\langle S \rangle$, then replacing it with an optimal ranking would lead to a ranking of $G$ using fewer colors than $\tau$ does, a contradiction. Thus $\tau$ uses $\td(G\langle T \rangle)$ and $\td(G-T)$ distinct colors on $T$ and $V(G)-T$, respectively, and thus $\td(G) = \td(G\langle T \rangle) + \td(G-T)$.
\end{proof}

\textit{Proof of Theorem~\ref{thm:starclique}}. Let $H$ be the graph obtained from $G$ by performing a star-clique transform at vertex $v$. The claim follows immediately by letting $T=V(G)-\{v\}$ and noting that $H=G\langle T \rangle$.\hfill $\Box$

Given an edge $e$ of $G$, let $G \cdot e$ denote the graph obtained from $G$ when edge $e$ is contracted.

\begin{theorem} \label{thm:endpointsnotspecial} Let $e\in E(G)$.  If $\td(G\cdot e)= \td(G)$, then the endpoints of $e$ are not $1$-unique.
\end{theorem}

\begin{proof} Let $e=uv$ and let $H$ be the graph obtained from a star-clique transformation at $u$ in $G$.  Note that $G\cdot e$ is isomorphic to the graph obtained from $G$ by deleting $u$ and adding edges from $v$ to $N_G(u)$.  Thus $G\cdot e$ is a subgraph of $H$ and $\td(G)=\td(G\cdot e)\leq \td(H)$.  By Theorem \ref{thm:starclique}, $u$ is not $1$-unique.  Similarly $v$ is not $1$-unique.
\end{proof}

To illustrate the utility of Theorems~\ref{thm:starclique}, \ref{thm:newbound}, and~\ref{thm:endpointsnotspecial}, we provide two examples. Theorem~\ref{thm:newbound} will later be important in the proof of Theorems~\ref{thm:s+r} and~\ref{thm:specialconstruction}.

\begin{example}\label{ex:1uniquesubgraphminimal}  \rm  Recall the graph $G_k$ that was defined in the previous section as cycle on $2^k+1$ vertices with a single triangular chord.   To establish the tree-depth of $G_k$ we use the following facts.
\begin{itemize}
\item (Katchalski et al.~\cite{path})  $\td(P_n)=\lfloor \log_2 n\rfloor +1$, for $n\geq 1$.
\item (Bruoth and Hor\u{n}\'ak~\cite{cycle})   $\td(C_n)=\lfloor \log_2 (n-1)\rfloor +2$, for $n\geq 3$.
\end{itemize}
By minor inclusion and the second fact, $\td(G_k)\geq \td(C_{2^k+1})=k+2$.  Let $v$ be a vertex of degree 3 in $G_k$.  Note that $G_k-v$ is a path on $2^k$ vertices.  By the first fact above there exists a ranking of $G_k-v$ using $k+1$ colors.  Using the same colors in $G_k$ and coloring vertex $v$ with $k+2$ shows that $\td(G_k)\leq k+2$.

Observe that a star-clique transform at a vertex of $G_k$ yields a graph $H$ isomorphic to either $C_{2^k}$ or $C_{2^k}$ with a chord.  Deleting a vertex of maximum degree in $H$ yields $P_{2^k-1}$.  From the first fact above, $\td(P_{2^k-1})=k$ and so as above we may add the deleted vertex back to create a ranking for $H$ using $k+1$ colors.  Since $\td(H)<\td(G_k)$ regardless of the vertex chosen, Theorem \ref{thm:starclique} implies that $G_k$ is $1$-unique. \hfill $ \Box$
\end{example}

\begin{example} \label{ex:subgraphNot1unique}\rm From the facts cited in Example \ref{ex:1uniquesubgraphminimal}, $\td(C_{2^k+2})=k+2$ and $\td(P_{2^k+2})=k+1$.  Thus as stated in the previous section $C_{2^k+2}\in \obs\sub(\mathcal{G}_{k+1})$.  Note that contracting an edge of $C_{2^k+2}$ yields $C_{2^k+1}$ and $\td(C_{2^k+2})=\td(C_{2^k+1})$.  Thus by Theorem \ref{thm:endpointsnotspecial}, $C_{2^k+2}$ is not $1$-unique. \hfill $ \Box$
\end{example}

As mentioned in the previous section and illustrated in Examples~\ref{ex:1uniquesubgraphminimal} and~\ref{ex:subgraphNot1unique}, the classes $\obs_{\subseteq} \mathcal{G}_k$ and $\mathcal{U}_{k+1}$ are incomparable under the subset relation. We now prove Theorem \ref{thm:conjecture}, which deals with the intersection of these classes.

{\it Proof of Theorem \ref{thm:conjecture}}.  If $G\in \mathcal{U}_{k+1}$, then by Theorem \ref{thm:endpointsnotspecial} contracting any edge of $G$ decreases the tree-depth.  If additionally $G\in \obs\sub(\mathcal{G}_k)$, then deleting any edge of $G$ decreases the tree-depth.  Thus $G$ is minor-minimal with tree-depth $k+1$. \hfill $\Box$.

Theorem \ref{thm:conjecture} shows that $1$-unique graphs that are also subgraph-minimal for their tree-depth are in fact minor-minimal.  Thus $1$-unique graphs differ from critical graphs by at most some additional edges.

\begin{observation}\label{obs:stillspecial} Let $G$ be $1$-unique.  If $e\in E(G)$ and $\td(G-e)=\td(G)$, then $G-e$ is $1$-unique.
\end{observation}

\begin{proof} Any ranking of $G$ in which a single vertex has rank 1 is also a ranking for $G-e$.
\end{proof}

\begin{theorem}\label{thm:spanningsubgraph} Every $1$-unique graph with tree-depth $k$ has a $k$-critical spanning subgraph.
\end{theorem}

\begin{proof} Let $G$ be a $1$-unique graph with tree-depth $k$. Iteratively delete edges whose removal does not decrease the tree-depth until this is no longer possible.  Let $H$ be the resulting graph.  By Observation \ref{obs:stillspecial}, $H$ is $1$-unique.   By Theorem \ref{thm:conjecture}, $H$ is $k$-critical.  Since $H$ has the same vertex set as $G$, it is a spanning subgraph.
\end{proof}

Note that in some sense, Theorem \ref{thm:spanningsubgraph} states that the converse of Conjecture \ref{conj: 1-unique} is almost true, further suggesting strong ties between 1-uniqueness and criticality.

\section{A construction for critical graphs} \label{sec:families}

The edge-addition result in Theorem~\ref{thm: DGT edge} allows us to construct critical graphs with arbitrarily large tree-depth. In this section we extend Theorem~\ref{thm: DGT edge} by using the property of $1$-uniqueness. We then show that Conjectures~\ref{conj: critical graph orders} and~\ref{conj: 1-unique} hold for all graphs generated by our construction.

The construction is as follows: henceforth let $H$ be an $s$-critical graph with vertices $v_1, \ldots, v_q$, and let $L_1,\ldots L_q$ be $(r+1)$-critical graphs.  Form a graph $G$ by choosing a vertex $w_i$ from each $L_i$ and identifying $v_i$ and $w_i$ for all $i\in \{1,\ldots,q\}$. (We say that the graphs $L_i$ are \emph{adjoined} at the vertices $v_i$ of $H$.) In the following results we show that $G$ has the properties we desire.

\begin{theorem}\label{thm:s+r}
The graph $G$ satisfies $\td(G)=r+s$.
\end{theorem}

\begin{proof} By Theorem~\ref{thm:newbound},
\begin{multline*}
\td(G) \leq \td\left(G\langle V(H) \rangle\right) + \td\left(G-V(H)\right)\\
 = \td(H) + \td\left((L_1-w_1)+\dots+(L_q-w_q)\right) = r+s.
\end{multline*}

We prove $\td(G)\geq r+s$ by induction on $s$.  When $s=1$, we have $G=L_1$ and $\td(G)=r+1$.  If $s>1$, consider an optimal ranking for $G$ and suppose the vertex with highest rank is in $L_i$.  Since $H$ is $s$-critical, $H-v_i$ has a $(s-1)$-critical minor $M$.  Consider the sequence of contractions and deletions that take $H$ to $M$.  These same operations performed on $G-V(L_i)$ produce a graph that has as a minor $M$ with one of $L_1,\ldots, L_{i-1}, L_{i+1}, \ldots, L_q$ adjoined at each vertex as before. By the induction hypothesis, this proper minor of $G$ has tree-depth at least $r+s-1$.  Thus $\td(G)\geq r+s$.
\end{proof}

\begin{theorem}\label{thm:specialconstruction} If the graphs $H$ and $L_i$ for $i=\{1,\ldots, q\}$ are 1-unique, then $G$ is 1-unique and $(r+s)$-critical.  Furthermore, if $|V(H)|\leq 2^{\td(H)-1}$ and $|V(L_i)|\leq 2^{\td(L_i)-1}$ for each $i$, then $|V(G)|\leq 2^{\td(G)-1}$.
\end{theorem}

\begin{proof} Assume that the graphs $L_i$ for $i=\{1,\ldots, q\}$ are 1-unique. We know from the previous theorem that $\td(G)=r+s$.  Pick an arbitrary vertex $u$ and suppose that $u\in V(L_j)$. Let $G'$ and $L'_j$ be the graphs resulting from a star-clique transform at $u$ in $G$ and in $L_j$, respectively, and let $H'$ be the graph resulting from a star-clique transform at $v_j$ in $H$. By Theorem~\ref{thm:newbound},
\begin{multline*}
\td(G') \leq \td\left(G'\langle V(H)-\{v_j\} \rangle\right) + \td\left(G'-\left(V(H)-\{v_j\}\right)\right)\\
 = \td(H') + \td\left(L'_j+(L_1-w_1)+\dots+(L_q-w_q)\right) = r+s-1,
\end{multline*}
where we omit $(L_j-w_j)$ in the disjoint union $(L_1-w_1)+\dots+(L_q-w_q)$ on the second line. By Theorem~\ref{thm:starclique}, $G$ is $1$-unique.

To show that $G$ is $(r+s)$-critical it is sufficient to consider contracting or deleting a single edge $e$. Let $G'$ be the resulting graph.

If $e$ is an edge of $H$, then the vertices of $G'$ corresponding to $H$ induce a subgraph with tree-depth $s-1$. By Theorem~\ref{thm:newbound},
\begin{multline*}
\td(G') \leq \td\left(G'\langle V(H) \rangle\right) + \td\left(G'-V(H)\right)\\
= \td(H)-1 + \td\left((L_1-w_1)+\dots+(L_q-w_q)\right) = r+s-1.
\end{multline*}
If $e$ is an edge of some $L_j$, then the vertices of $G'$ corresponding to $H-v_j$ induce a subgraph with tree-depth $s-1$, and the vertices of $G'$  corresponding to $L_j-e$ induce a subgraph with tree-depth $r$, so by Theorem~\ref{thm:newbound},
\begin{multline*}
\td(G') \leq \td\left(G'\langle V(H)-\{v_j\} \rangle\right) + \td\left(G'-(V(H)-\{v_j\})\right)\\
= \td(H)-1 + r = r+s-1.
\end{multline*}
Thus $G$ is $(r+s)$-critical.

Suppose $|V(H)|\leq 2^{\td(H)-1}=2^{s-1}$ and $|V(L_i)|\leq 2^{\td(L_i)-1}=2^{r}$ for each $i$.  Since  every vertex of $G$ belongs to some $L_i$ with $1\leq i\leq |V(H)|$ we have $|V(G)|\leq |V(H)|\max_i |V(L_i)| \leq 2^{s-1}2^{r} = 2^{\td(G)-1}$.
\end{proof}

\begin{example}\label{ex:contruction}  \rm The graph in Figure \ref{fig:construction} is constructed from a $4$-critical graph (shown with shaded vertices) by adjoining copies of the 3-critical graphs $P_4$ and $K_3$.  By Theorem \ref{thm:specialconstruction}, the graph is $6$-critical.%
\begin{figure}[htb]
       \center{\includegraphics[scale=0.3]{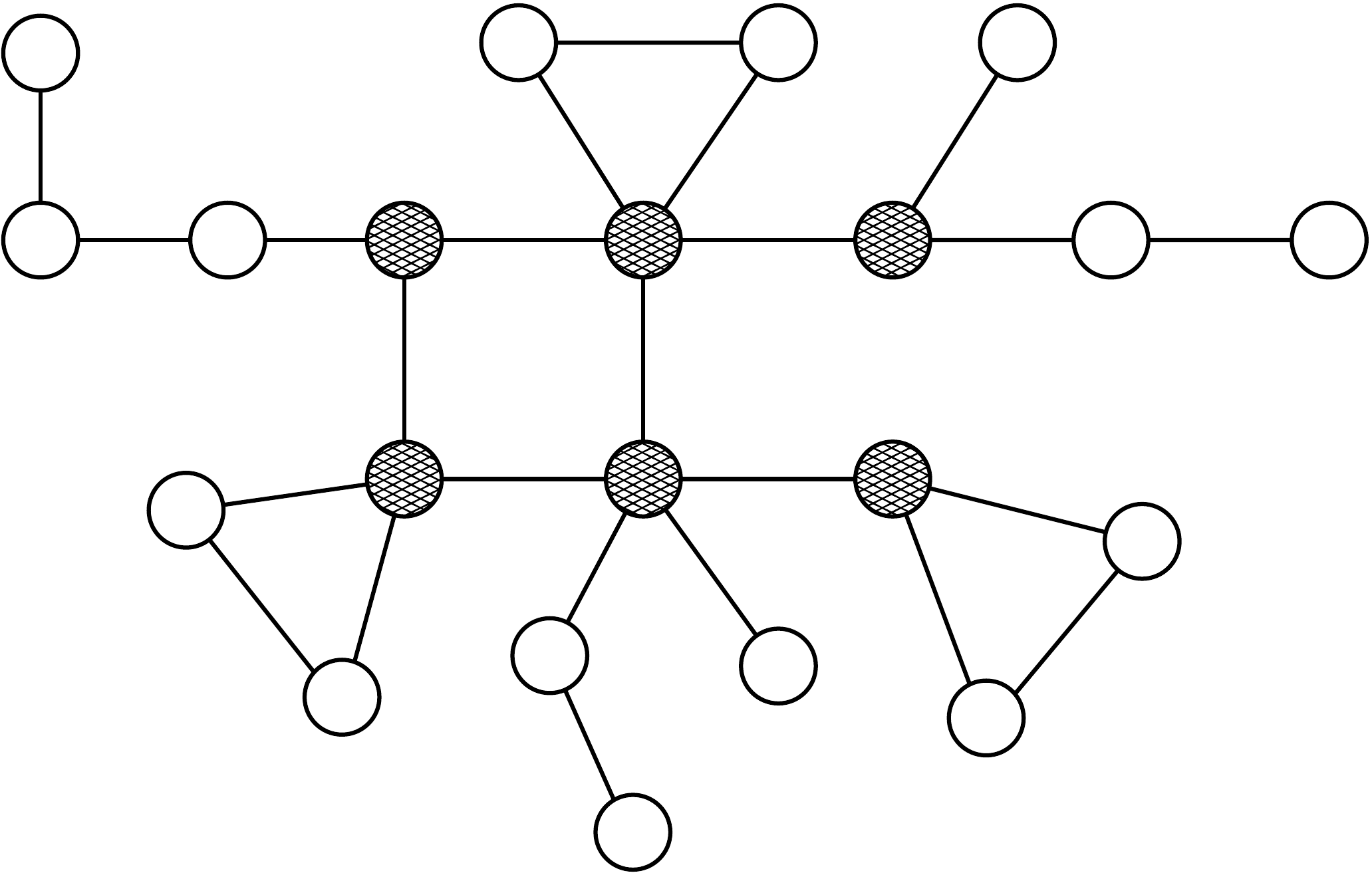}}
              \caption{A graph in $\mathcal{S}_6$}\label{fig:construction}
\end{figure}
\phantom{.}\hfill $ \Box$
\end{example}

Using our construction, we may inductively construct a large family of critical graphs with tree-depth $k$, given a family $\mathcal{F}$ of graphs already determined to be critical and 1-unique. Let $\mathcal{S}_2 = \{K_2\}$, and for $k > 2$ define $\mathcal{S}_k$ to be the family consisting of all the $k$-critical graphs in $\mathcal{F}$, together with all graphs $G$ that may be constructed as above with $H$ taken from $\mathcal{S}_s$ and the $L_i$'s taken from $\mathcal{S}_{r+1}$ such that $r\geq 1$, $s \geq 2$, and $r+s = k$.

By the results above, $\mathcal{S}_k$ is a family of $k$-critical graphs that are all 1-unique. Furthermore, if Conjecture~\ref{conj: critical graph orders} holds for all elements of $\mathcal{F}$, then by Theorem~\ref{thm:specialconstruction} it holds for all elements of $\mathcal{S}_k$. Clearly the size of the class $\mathcal{S}_k$ depends on the family $\mathcal{F}$, and it is an interesting task to find new infinite families of critical, 1-unique graphs that may be included in $\mathcal{F}$. Using the techniques similar to those in Examples~\ref{ex:1uniquesubgraphminimal} and~\ref{ex:subgraphNot1unique}, it is possible to show~\cite{unpub} that each of the following graphs is $k$-critical and 1-unique and satisfies Conjecture~\ref{conj: critical graph orders}:

\begin{itemize}
\item For each $k \geq 1$ and $s \in \{1,\dots,k\}$, a graph $Q$ obtained in the following way: Let $H_0$ be a complete graph with $s$ vertices, and for some $q \in \{1,\dots,s\}$, let $H_1, \dots, H_q$ be vertex-disjoint complete graphs, each with $k-s$ vertices. Given a partition $\pi_1+\dots+\pi_q$ of $s$ into positive integers, choose a partition $B_1,\dots,B_q$ of the vertices of $H_0$ so that $|B_i|=\pi_i$ for all $i$, and form $Q$ by adding to the disjoint union $H_0+H_1+\dots+H_q$ all possible edges between vertices in $B_i$ and vertices of $H_i$, for all $i$. (When $s=1$ or $s=k$, or when $q=1$, the graph $Q$ is isomorphic to $K_k$. Note that graph in Figure~\ref{fig:minorminimallist} with degree sequence $(3,3,3,2,1)$ also has this form.)

\item For each $k \geq 3$ and $t$ such that $0 \leq t \leq 2^{k-2}-2$, the graph $R_{k,t}$ obtained by taking a path with $2^{k-2}+1+t$ vertices and adding an edge between the two vertices at distance $t$ from the endpoints. (Note that $R_{k,0}=C_{2^{k-2}+1}$, and Figure~\ref{fig:minorminimallist} contains $R_{3,0}$, $R_{4,0}$, $R_{4,1}$, and $R_{4,2}$.)
\end{itemize}

An open question is whether it is possible to find a class $\mathcal{F}$ with a simple and nontrivial description such that every (1-unique) $k$-critical graph belongs to $\mathcal{S}_k$.

Alternatively, perhaps the construction may be generalized. Note that graphs generated by our construction are formed by ``overlapping'' smaller critical, 1-unique graphs (the graphs $L_i$) on vertices of a central graph that is also critical (the graph $H$). The graph $Q$ in the example above shares a similar property; it may be considered as the result of overlapping complete graphs $K_{|B_i|+k-s}$ on the complete graph $H_0$. 

However, if, in attempting to apply the construction, the vertices of a single $L_i$ are carelessly identified with more than one vertex of $H$, the graph $G\langle V(H) \rangle$ may no longer be critical or have other properties that allow us to ensure that $G$ is critical with a desired tree-depth. Still, motivated by the example $Q$ above, we may suspect that under the right conditions we can maintain criticality or 1-uniqueness while identifying multiple vertices in each $L_i$ with vertices in $H$. We leave it as an open question to determine these conditions.

\bibliographystyle{elsarticle-num}
\bibliography{treedepth}

\end{document}